\documentclass[11pt]{amsart}
\usepackage{amsmath,amssymb,amsfonts}
\usepackage{graphicx,xcolor}
\usepackage{epstopdf}
\usepackage[colorlinks=true]{hyperref}
\hypersetup{urlcolor=blue, citecolor=red}

\usepackage[all,cmtip]{xy}

\usepackage{tikz}

\usepackage[left=2.5cm,top=2.5cm,right=2.5cm,bottom=2.5cm]{geometry}

\newtheorem{theorem}{Theorem}[section]

\newtheorem{lemma}[theorem]{Lemma}

\newtheorem*{problem}{Problem}

\theoremstyle{definition}

\newtheorem{remark}{Remark}

\usepackage{cancel}
\usepackage[normalem]{ulem}
\usepackage{tkz-graph}
\usetikzlibrary{arrows}
\usepackage{caption}

\title[Backward Julia sets for a class of p-adic Hénon like maps] 
      {Backward Julia sets for a class of p-adic Hénon like maps}

\author[Jéfferson Bastos, Danilo Caprio, Oyran Raizzaro]{}


\email{jefferson.bastos@unesp.br, danilo.caprio@unesp.br, oyran@uems.br }

\begin{document} 
\maketitle

\centerline{\scshape Jéfferson L. R. Bastos}
\medskip
{\footnotesize
	\centerline{UNESP - Departamento de Matem\'atica do Instituto de Bioci\^encias, Letras e Ci\^encias Exatas.}
	\centerline{São Paulo - Brazil}}

\bigskip

\centerline{\scshape Danilo Caprio}
\medskip
{\footnotesize
	\centerline{UNESP - Departamento de Matem\'atica da Faculdade de Engenharia.}
	\centerline{São Paulo - Brazil}}

\bigskip

\centerline{\scshape Oyran Raizzaro}
\medskip
{\footnotesize
	\centerline{UEMS - Departamento de Matem\'atica da Universidade Estadual de Mato Grosso do Sul.}
	\centerline{Mato Grosso do Sul - Brazil.}

}

\bigskip


\begin{abstract}
In this work we study the backward filled Julia sets of a class of $p$-adic polynomial maps $f:\mathbb{Q}_p^2\longrightarrow \mathbb{Q}_p^2$ defined by $f(x,y)=(xy+c,x)$, where $c\in\mathbb{Q}_p$ is a $p$-adic number. In particular, if $|c|\leq 1$, then we proved that the backward filled Julia set of $f$ is a bounded subset in $\mathbb{Z}_p^2$. On the other hand, if $|c|> 1$, then we prove that the backward filled Julia set of $f$ is an unbounded set and has infinity Haar measure.
\end{abstract}

\section{Introduction}

Let $f:\mathbb{C}\longrightarrow \mathbb{C}$ be a holomorphic map.
The forward filled Julia set associated to $f$ (or simply filled Julia set) is by definition the set
$$\mathcal{K}^+(f)=\{z\in\mathbb{C}: (f^n(z))_{n\geq 0} \textrm{ is bounded}\},$$ where $f^n$
is the $n-$th iterate of $f$. Also, the backward filled Julia set associated to $f$ is by definition the set
$$\mathcal{K}^-(f)=\{z\in\mathbb{C}: f^{-n}(z) \textrm{ there exists for all }n \in \mathbb{Z}_+ \textrm{ and }(f^{-n}(z))_{n\geq 0} \textrm{ is bounded}\}.$$ These definitions can be extended to polynomial maps defined in $\mathbb{C}^d$ for $d\geq 2$.

Filled Julia sets and their boundaries (called Julia set) have many topological and dynamical properties.
These sets were defined independently by Fatou \cite{fatou19,fatou21} and Julia  \cite{julia18,julia22} and they are associated with many areas such as dynamical systems, number theory, topology and functional analysis (see \cite{dv}).

Julia sets in higher dimensions were studied by many mathematicians. For example, an interesting class of Julia sets was  studied in \cite{g}. Precisely, the author considered a family of maps $f_{\alpha,\beta}:\mathbb{C}^2\longrightarrow\mathbb{C}^2$ defined by $(x,y)\mapsto(xy+\alpha,x+\beta)$ with $\alpha,\beta\in\mathbb{C}$ and he proved, among other things, that the family $f_{\alpha,\beta}$ has a measure of the maximal entropy $\frac{1+\sqrt{5}}{2}$.  In \cite{ms}, the authors considered $f_\alpha=f_{\alpha,0}$ and they  extended the Killeen and Taylor stochastic adding machine in base $2$ (see \cite{killentaylor}) to the Fibonacci base and they proved that the spectrum of the transition operator associated with this stochastic adding machine is related to the set $\mathcal{K}^+(f_\alpha)$, where $\alpha\in\mathbb{R}$ is a real value.  The real case $f_{\alpha}:\mathbb{R}^2\longrightarrow\mathbb{R}^2$ was studied in \cite{BCM} (for {$0< \alpha<1/4$}) and \cite{C} (for $-1< \alpha< 0$) where the authors proved that $\mathcal{K}^+(f_\alpha)$ is the union of stable manifolds of the fixed and $3$-periodic points of $f_\alpha$. Also, they proved that $\mathcal{K}^-(f_\alpha)$ is the union of unstable manifolds of the saddle fixed and $3$-periodic points of $f_\alpha$.

A more general class is given in \cite{cmv}, where the authors defined  the stochastic Vershik map related to a stationary ordered Bratteli diagram with incidence matrix $\left(\begin{array}{cc} a & b \\ c & d\end{array} \right)$, where $a,b,c,d\in\mathbb{N}$, and they proved that the spectrum of the transition operator associated with this is connected to the set $\mathcal{K}^+(f_{\alpha,a,b,c,d})$, where $f_{\alpha,a,b,c,d}:\mathbb{C}^2\longrightarrow\mathbb{C}^2$ is a family of maps defined by $(x,y)\mapsto(x^ay^b+\alpha,x^cy^d+\alpha)$ with $\alpha\in\mathbb{R}$.

Another interesting class of Julia sets is associated to the H\'enon maps: let $f:\mathbb{C}^2\longrightarrow \mathbb{C}^2$ defined by
$f(x,y)=(y,P(y)+c-ax)$, for all $(x,y)\in \mathbb{C}^2$, where $P$ is a complex polynomial function and $a,c$ are fixed complex numbers.
Let us observe that H\'enon maps and their dynamics on $\mathbb{R}^2$ and $\mathbb{C}^2$ have been extensively studied, see for instance \cite{bs,BC,fs,hw,H}. For example, in \cite{hw} the authors considered the map $H_a:\mathbb{R}^2\longrightarrow \mathbb{R}^2$ defined by $H_a(x,y)=(y,y^2+ax)$ where $0<a<1$ is given and they proved  that $\mathcal{K}(H_a)=\{\alpha,p\}\cup[W^s(\alpha)\cap W^u(p)]$, where $\alpha=(0,0)$ is the attracting fixed point of $H_a$, $p=(1-a,1-a)$ is the repelling fixed point of $H_a$, $\mathcal{K}(H_a):=\mathcal{K}^+(H_a)\cap \mathcal{K}^-(H_a)$ is the Julia set associated to $H_a$ and $W^s(\alpha)$ and $W^u(p)$ are the stable and unstable manifolds of $\alpha$ and $p$, respectively.

On the other hand, the study of dynamical properties in the context of $p$-adic dynamics is developing and it has applications in many areas, such as physics, cognitive
science, cryptography, biology and geology among others (see for instance \cite{A,KN,KOL,Y}). Hence, Hénon maps were also studied in the context of $p$-adic dynamics (see for instance \cite{ADP,AV,WS}). For example, in  \cite{ADP} the authors studied dynamical and topological properties of the map $H_{a,b}:\mathbb{Q}^2_p\longrightarrow \mathbb{Q}^2_p$ defined by $H_{a,b}(x,y)=(a+by-x^2,x)$, where $a,b\in\mathbb{Q}_p$ with $b\neq 0$. In particular, the authors established basic properties of its one-sided and two-sided filled Julia sets, and they determined, for each pair of parameters $a,b\in\mathbb{Q}_p$, whether these sets are empty or nonempty, whether they are bounded or unbounded, whether they are equal to the unit ball or not, and for
a certain region of the parameter space they showed that the filled Julia set is an attractor. Also, for a certain region of the parameter space, they showed that the Hénon map is topologically conjugate on its filled Julia set to the two-sided shift map on the space of bisequences on two symbols,  which is related to the Smale horseshoe map.

Let $f=f_c:\mathbb{Q}^2_p\longrightarrow\mathbb{Q}^2_p$ be the map defined by $f(x,y)=(xy+c,x)$, with $c\in\mathbb{Q}_p$. In \cite{BCR}, the authors studied the forward  Julia set $\mathcal{K}^+$ defined by $$\mathcal{K}^+=\mathcal{K}^+(f)=\{(x,y)\in\mathbb{Q}_p^2: (f^n(x,y))_{n\geq 0} \textrm{ is bounded}\}.$$ 
In particular, if $ |c|_p\leq 1$, then they have proved that $\mu_2(\mathcal{K}^+)=+\infty$, where $\mu_2$ is the Haar measure defined on the Borel $\sigma$-algebra of $\mathbb{Q}_p^2$ and if $(x,y)\in\mathcal{K}^+$, then there exists a nonnegative integer $k$ such that $f^n(x,y)\in \mathbb{Z}_p^2$ for all $n\geq k$, where $\mathbb{Z}_p$ is the set of $p$-adic integers and $|\cdot|_p$ is the p-adic norm defined on $\mathbb{Q}_p$. Furthermore, if $|c|_p<1$ and $(x,y)\in \mathbb{Z}_p^2$ then $f^n(x,y)$ converges to a fixed point of $f$.  On the other hand, if $|c|_p>1$, then they have proved that there exists $R>0$ such that $\mathcal{K}^+$ is the set characterized by the points $(x,y)$ in $\mathbb{Q}_p^2$ whose the orbits $(x_n,y_n)=f^n(x,y)$ never satisfy $\|(x_n,y_n)\|_p:=\max\{ |x_n|_p, |y_n|_p\}>R$, for all $n\geq 0$.

Motivated by previous studies of this map, as well as investigations of Julia sets of Hénon maps defined over the $p$-adic numbers $\mathbb{Q}_p^2$, our goal is to describe the backward $p$-adic Julia set of $f$. It is worth noting that the polynomial $f$ does not have a constant Jacobian, which distinguishes it from a genuine Hénon map. For this reason, we refer to $f$ as a Hénon-like map.

It is easy to check that the only fixed points of $f$ in $\mathbb{Q}_p^2$ are $(a,a)$, where $a$ is a root of $$x^2-x+c=0.$$ Obviously, $f$ has no fixed points in $\mathbb{Q}_p^2$ if $1-4c$ is not a square in $\mathbb{Q}_p$ and a single fixed point in $\mathbb{Q}_p^2$ given by $(\frac{1}{2},\frac{1}{2})$ if $1-4c=0$. On the other hand, if $1-4c$ is a nonzero square in $\mathbb{Q}_p$, then $f$ has two distinct fixed points in $\mathbb{Q}_p^2$ given by $\alpha=(\alpha_1,\alpha_1)$ and $\gamma=(\alpha_2,\alpha_2)$ with
\begin{center}
	$\alpha_1=\dfrac{1-q}{2}$ and $\alpha_2=\dfrac{1+q}{2}$
\end{center}
where $q\in\mathbb{Q}_p$ satisfies $q^2=1-4c$. Furthermore, $\rho=(-1,-1)$, $f(\rho)=(1+c,-1)$ and $f^2(\rho)=(-1,1+c)$ is a $3-$cycle of $f$ in $\mathbb{Q}_p^2$.

In this work, we study the backward Julia set defined by $$\mathcal{K}^-=\mathcal{K}^-(f)=\{(x,y)\in\mathsf{Q}: (f^{-n}(x,y))_{n\geq 1} \textrm{ is bounded}\}$$ where $$\mathsf{Q}=\{(x,y)\in\mathbb{Q}^2_p: f^{-n}(x,y) \textrm{ exists for all } n \in\mathbb{Z}_+\}. $$ 

Here, we characterize the backward $p$-adic Julia set of $f$. In particular, if $|c|_p \leq 1$, we prove that $\mathcal{K}^-$ is a bounded subset of $\mathbb{Z}_p^2$. On the other hand, if $|c|_p > 1$, we show that $\mathcal{K}^-$ is an unbounded subset of $\mathbb{Q}_p^2$ and that $\mu_2(\mathcal{K}^-) = +\infty$, where $\mu_2$ denotes the Haar measure on the Borel $\sigma$-algebra of $\mathbb{Q}_p^2$. Furthermore, we exhibit a bounded subset $\mathsf{C}_0 \cup \mathsf{J}_0 \subset \mathbb{Q}_p^2$ satisfying the following property: if $(x,y) \in \mathcal{K}^-$, then there exists $k \in \mathbb{Z}_+$ such that $f^{-n}(x,y) \in \mathsf{C}_0 \cup \mathsf{J}_0$ for all $n \geq k$. It is worth mentioning that the proofs in this paper rely on studying the dynamics of $f$ under a convenient partition of $\mathbb{Q}_p^2$, which allows us to identify the subsets of the partition that are contained in $\mathcal{K}^-$ and those that are not.

The paper is organized as follows. In Section \ref{preli} we provide some preliminaries about $p$-adic numbers. Section \ref{mainresults} is devoted to stating the main results of the paper. In Sections \ref{provakmenosmenor}, \ref{provateokmenosmaior} and \ref{provakmenosigual} we give the proof of Theorems \ref{theokmenosmenor}, \ref{teokmenosmaior} and \ref{teokmenosum}, respectively. In Section \ref{open}, we state some open problems arising from this work.


\section{Preliminaries} \label{preli}

In this section we provide some preliminaries about $p$-adic numbers, most of which can be found in \cite{BCR}, but are included here for completeness and the reader’s convenience.

Let $\mathbb{N}$ be the set of nonnegative integers.
The additive group of $p$-adic integers can be seen as the set $\mathbb{Z}_p$ of sequences $(x_i)_{i \in \mathbb{N}}$
where $x_i$ is an integer $mod \; p^{i+1}$.
Each element of  $\mathbb{Z}_p$  can be formally written as a series
$$\sum_{i=0}^{+\infty} a_i p^i \mbox { where  } a_i \in A= \{0,\ldots, p-1 \}.$$

Let $x= \sum_{i=0}^{+\infty} a_i p^i$ and $y= \sum_{i=0}^{+\infty} b_i p^i $ be  two elements of $\mathbb{Z}_p$. We turn  $\mathbb{Z}_p$ into a compact metric space by  considering the distance
$$d(x, y)= \left\{\begin{array}{cl}
	0 & \mbox { if } a_i= b_i \mbox { for all } i \geq 0, \\
	p^{-j} & \mbox { otherwise, where } j= \min\{i \geq 0;\; a_i \ne b_i\} .
\end{array} \right. $$
We denote
$$| x- y |_p= d(x, y)$$  and we have $$| x+ y |_p \leq \max \{| x |_p, |  y |_p\}$$ 
where $$| x+ y |_p = \max \{| x |_p, |  y |_p\} \textrm{ if } |x|_p \neq |y|_p.$$  

We can prove that $\mathbb{Z}_p$ is the completion of $\mathbb{Z}$ under the metric $d$.

We can also add $x$ to $y$,  by adding coordinate-wise and if any of the sums is $p$ or more, we obtain the sum modulo $ p$ and we take a carry of 1 to the next sum. We can also multiply
elements of $ \mathbb{Z}_p$. With this, we turn $(\mathbb{Z}_p, +,\; \cdot \;)$ into a ring.

The set $\mathbb{Z}_p$ can also be seen as  the projective limit of $(\mathbb{Z}/p^n \mathbb{Z}, \pi_n),\; n \geq 1 $, where for all positive integer $n$,
$\mathbb{Z}/p^n \mathbb{Z}$ is endowed with the discrete topology, $\prod_{n=1}^{+\infty}\mathbb{Z}/p^n \mathbb{Z}$ equipped with the product topology and $ \pi_n :  \mathbb{Z}/p^{n+1} \mathbb{Z} \to  \mathbb{Z}/p^n \mathbb{Z},\;  n \geq 1,$ is the canonical homomorphism.   This limit can be described as
$$\mathbb{Z}_p= \{ (x_n)_{n \geq 1} \in \prod_{n=1}^{+\infty}\mathbb{Z}/p^n \mathbb{Z},\; \pi_n (x_{n+1})= x_n,\; \forall n \geq 1\}.$$
With this definition $\mathbb{Z}_p$ is a compact set of
$\prod_{n=1}^{+\infty}\mathbb{Z}/p^n \mathbb{Z}$ (see \cite{Ro}).
$\mathbb{Z}_p$ is a Cantor set  homeomorphic to $\Pi_{n=1}^{+\infty}A$  endowed with the  product topology of the discrete topologies on $A$.
Through this homeomorphism, any point $x= (x_n)_{n \geq 1}$ can be represented as $x= \sum_{i=0}^{+\infty} a_i p^i \mbox { where  } a_i \in A= \{0,\ldots, p-1 \}$
and $x_n = \sum_{i=0}^{n-1} a_i p^i.$

For any integer $n \geq 1, \; \mathbb{Z}/p^n \mathbb{Z}$
can be seen as a subset of $\mathbb{Z}_p$ and any element
$x $ in $\mathbb{Z}/p^n \mathbb{Z}$ can be represented  as $x=  \sum_{i=0}^{n-1} a_i p^i$ where  $a_i \in A$ for all $0 \leq i \leq n-1$.

\vspace{0.5em}

The field  of fractions of $\mathbb{Z}_p$ is denoted by $\mathbb{Q}_p$.
We have
$$ \mathbb{Q}_p= \mathbb{Z}_p [1/p]= \bigcup_{i \geq 0}p^{-i}\mathbb{Z}_p.$$
Any element  $x$ of $\mathbb{Q}_p \setminus \{0\}$ can be written as a
$$x= \sum_{i=l}^{+\infty} a_i p^i \mbox { where  } a_i \in A= \{0,\ldots, p-1 \},\; l \in \mathbb{Z} \mbox { and } a_l \neq 0.$$
On $ \mathbb{Q}_p$, we define   the absolute value $| \cdot |_p$ by
$|x |_p= p^{-l}$ and the $p$-adic metric $d$ by
$d(x, y)= |x - y|_p$ for all $x, y$ in $ \mathbb{Q}_p$.

We can prove that $\mathbb{Q}_p$ is the completion of $\mathbb{Q}$ under the $p$-adic metric. Moreover,  $\mathbb{Q}_p$  is not compact but is a locally compact set and $\mathbb{Z}_p$ is the closed unit ball of $ \mathbb{Q}_p$.

Finally, it is known that the Haar measure on a locally compact abelian group is a translation-invariant measure that is unique up to multiplication by a constant. Since $\mathbb{Q}_p$ is a locally compact abelian group under addition, there exists a unique Haar measure $\mu$ that is normalized so that $\mu(\mathbb{Z}_p)=1$. Under this normalization, the Haar measure of a closed ball in $\mathbb{Q}_p$ is equal to its radius, that is, $\mu(\overline{B}_r(a)) = r$, where $\overline{B}_r(a)=\{x\in\mathbb{Q}_p: |x-a|_p\leq r\}$ (for the studies about Haar measure under the set of $p$-adic numbers, see for instance \cite{furno} and \cite{zuniga}). Furthermore, the Haar measure $\mu_2$ on $(\mathbb{Q}_p^2,\mathcal{B}(\mathbb{Q}_p^2))$, where $\mathcal{B}(\mathbb{Q}_p^2)$ is the Borel $\sigma$-algebra of $\mathbb{Q}_p^2$, agrees with the product measure, that is $\mu_2=\mu \times \mu$.


\section{Main results} \label{mainresults}

Let $p\in \mathbb{N}$ be an odd prime number, $c\in \mathbb{Q}_p$ be a $p$-adic number and consider the non-Archimedean $p$-adic norm $|\cdot |=|\cdot |_p$ on $\mathbb{Q}_p$. Also, let $f=f_c:\mathbb{Q}^2_p\longrightarrow\mathbb{Q}^2_p$ be the map defined by $f(x,y)=(xy+c,x)$. In this work, we study the backward Julia set defined by $$\mathcal{K}^-=\mathcal{K}^-(f)=\{(x,y)\in\mathsf{Q}: (f^{-n}(x,y))_{n\geq 1} \textrm{ is bounded}\}$$ where $$\mathsf{Q}=\{(x,y)\in\mathbb{Q}^2_p: f^{-n}(x,y) \textrm{ exists for all } n \in\mathbb{Z}_+\}. $$ 

\begin{remark} \label{conjq}
Denoting the backward orbit of $(x,y)$ by $(x_{-n},y_{-n}) = f^{-n}(x,y)$, where $f^{-1}(x,y)=\left(y,\frac{x-c}{y}\right)$, we have that $(x,y)\in\mathsf{Q}$ if, and only if, $y_{-n}\neq 0$ for all $n\in\mathbb{Z}_+$, that is, $y\neq 0$ and $x_{-n}\neq c$ for all $n\in\mathbb{Z}_+$.
\end{remark}

For describing the first main result, it will be necessary to consider the following sets:
$$Z=\{(x,y)\in\mathsf{Q}:|x|=|y|=1\} \textrm{ and } R=\{(x,y)\in\mathsf{Q}:|x|=|y|=|c|\}.$$

\begin{theorem} \label{theokmenosmenor}
	If $|c|<1$ then $Z\subset \mathcal{K}^-\subset R\cup Z$.
\end{theorem}

\begin{remark} It is worth to mention that, if $c\in\mathbb{Q}_p$ is a $p$-adic number satisfying $|c|<1$, then there exist examples of points in $R\setminus \mathcal{K}^{-}$ as well as points in $R\cap \mathcal{K}^{-}$.

Firstly, for instance, let $c=p$, $x=p+2p^{3}$ and $y=2p$, with $p\neq 2$. Thus, $(x,y)\in R\setminus\mathcal{K}^-$. In fact, observe that
\begin{align*}
	(x_{-1}, y_{-1}) &= (2p, p^2) 
	\Rightarrow(|x_{-1}|, |y_{-1}|) = (p^{-1}, p^{-2}), \\
	(x_{-2}, y_{-2}) &= \left( p^2, 1/p \right) 
	\Rightarrow (|x_{-2}|, |y_{-2}|) = (p^{-2}, p)
\end{align*}
and continuing in this way, we have
\[
(|x_{-2n-1}|,|y_{-2n-1}|)=(p^{2^n-1},\,p^{-2^n}) \;\; \textrm{ and } \;\;
(|x_{-2n-2}|,|y_{-2n-2}|)=(p^{-2^n},\,p^{2^{n+1}-1}),
\]
for all $n\geq 1$.

Hence,
\begin{itemize}
	\item from Remark \ref{conjq}, we have $(x,y)\in\mathsf{Q}$ since $|y_{-n}|\neq 0$, for all $n\in\mathbb{Z}_+$;
	
	\item $(x,y)\in R$, since $|x|=|y|=p^{-1}=|c|$;
	
	\item $(x,y)\not\in \mathcal{K}^-$ since $\|(x_{-n},y_{-n})\|_p \xrightarrow{n\to+\infty } +\infty$.
\end{itemize}

On the other hand, to obtain a point in $R\cap K^{-}$, consider  $c = p - p^{2}$ and $(x,y) = (p,p)$. In this case $(x,y)\in R$, and a direct verification shows that the backward orbit is bounded.
\end{remark}

Let $(F_n)_{n\geq 0}$ be the Fibonacci sequence, that is $F_0=F_1=1$ and $F_n=F_{n-1}+F_{n-2}$, for all $n\geq 2$. Let $|c|>1$, and consider the following sets 
$$\mathsf{J}_0=\{(x,y)\in \mathsf{Q} ;\ |x|<|c|,\ 1<|y|<|c|\}, \;\; \textrm{ and } \;\; \mathsf{C}=\bigcup_{i=0}^{+\infty} (\mathsf{C}_i \cup \mathsf{D}_{i+2}), $$
where
$$\mathsf{C}_0:=\{(x,y)\in\mathsf{Q}: |x|<|c| \textrm{ and } (|y|=1 \textrm{ or } |y|=|c|)\} \cup \{(x,y)\in\mathsf{Q}: |x|=|c| \textrm{ and } |y|\leq|c|\},$$
$$\mathsf{C}_{2n+1}:=\{(x,y)\in\mathsf{Q}: |c|^{F_{2n}}<|x|\leq|c|^{F_{2n+2}} \textrm{ and } |y|=(|c|^{-1}|x|^{F_{2n}})^{1/F_{2n+1}} \}, \; \textrm{ for } n\geq 0,$$
$$\mathsf{C}_{2n+2}:=\{(x,y)\in\mathsf{Q}: |c|^{F_{2n+1}}<|x|\leq|c|^{F_{2n+3}} \textrm{ and } |y|=(|c||x|^{F_{2n+1}})^{1/F_{2n+2}} \}, \; \textrm{ for } n\geq 0,$$
$$\mathsf{D}_2:=\{(x,y)\in\mathsf{Q}: |c|<|x|<|c|^2 \textrm{ and } |y|=|c|\},$$
$$\mathsf{D}_{2n+1}:=\{(x,y)\in\mathsf{Q}: |c|^{F_{2n}}<|x|<|c|^{F_{2n+1}} \textrm{ and } |y|=(|c|^{-1}|x|^{F_{2n-2}})^{1/F_{2n-1}} \}, \; \textrm{ for } n\geq 1,$$
$$\mathsf{D}_{2n+2}:=\{(x,y)\in\mathsf{Q}: |c|^{F_{2n+1}}<|x|<|c|^{F_{2n+2}} \textrm{ and } |y|=(|c||x|^{F_{2n-1}})^{1/F_{2n}} \}, \; \textrm{ for } n\geq 1.$$

\begin{theorem} \label{teokmenosmaior}
If $|c|>1$, then

\medskip 

\noindent a) $\displaystyle \bigcup_{i=0}^{+\infty} f^n(\mathsf{J_0})\subset \mathcal{K}^- \subset \mathsf{C}\cup \bigcup_{i=0}^{+\infty} f^n(\mathsf{J_0}).$

\medskip 

In particular,

\medskip 

\noindent b) if $|c|=p$, then $\mathsf{J}_0 =\emptyset$ and $\mathcal{K}^- \subset \mathsf{C}$;

\medskip 

\noindent c) if $|c|>p$, then $\mu_2(\mathcal{K}^-)=+\infty$, where $\mu_2$ is the Haar measure defined on $(\mathbb{Q}_p^2,\mathcal{B}(\mathbb{Q}_p^2))$.
\end{theorem}

\begin{remark} If $|c|>1$, then there exist examples of points in $\mathsf{C}\setminus \mathcal{K}^{-}$ as well as points in $\mathsf{C}\cap \mathcal{K}^{-}$. 

For instance, let $c=\frac{1}{p}$ and $(x,y)=\left(\frac{1}{p}+p^{2},\,1\right)$. Since $|x|=|c|=p$ and $|y|<|c|$, we have $(x,y)\in \mathsf{C}_{0}\subset \mathsf{C}$.  Note that the first iterates satisfy
		\[
		(|x_{-1},|y_{-1}|)=(1,p^{-2}),\qquad
		(|x_{-2}|,|y_{-2}|)=(p^{-2},p^{3}),
		\]
		\[
		(|x_{-3}|,|y_{-3}|)=(p^{3},p^{-2}),\qquad
		(|x_{-4}|,|y_{-4}|)=(p^{-2},p^{5}),
		\]
		so the same alternating pattern continues, with the $p$-adic norms growing without bound, that is $
		\|f^{-n}(x,y)\|_p\longrightarrow +\infty$, and therefore $(x,y)\notin \mathcal{K}^{-}$.
		
		On the other hand, let $c=\frac{1}{p}$, and consider $(x,y)=(1,1)$. In this case $|x|=1<|c|=p$ and $|y|=1$, so $(x,y)\in \mathsf{C}_{0}\subset \mathsf{C}$.  
		A direct computation of the backward iterates shows that $
		\|(x_{-n},y_{-n})\|_p \le p \quad\text{for all } n\ge 0$. 
		Therefore, $(x,y)\in \mathsf{C}\cap \mathcal{K}^{-}$.
\end{remark}

\begin{theorem} \label{teokmenosum}
	If $|c|=1$, then  $\mathcal{K}^- \subset \mathsf{C}_0$, that is
	$$\mathcal{K}^- \subset \{(x,y)\in \mathsf{Q}: \; \max \{|x|,|y|\}=1\}.$$
\end{theorem}

\begin{remark}
If $|c|=1$, then there exist points in $\mathsf{C}_0\cap \mathcal{K}^{-}$ as well as points in $\mathsf{C}_0\setminus \mathcal{K}^{-}$, where $\mathsf{C}_0=\{(x,y)\in\mathsf{Q} : \max\{|x|,|y|\}=1\}$. For instance, let $c=1$. Hence,

\begin{itemize}
		
		\item $(x,y)=(-1,-p) \in \mathsf{C}_0\setminus \mathcal{K}^-$. In fact, we have $\max\{|-1|,|-p|\}=1$ and 
$$(x_{-1}, y_{-1}) = (-p, 2/p) 
	\Rightarrow(|x_{-1}|, |y_{-1}|) = (p^{-1}, p).$$
Thus,
		\[
		(|x_{-2n}|,|y_{-2n}|)=(p^{2^{n-1}},\,p^{-2^{n-1}}) \;\; \textrm{ and } \;\;
		(|x_{-2n-1}|,|y_{-2n-1}|)=(p^{-2^{n-1}},\,p^{2^{n}}),
		\]
		for all $n\geq 1$.
		
		Hence,
		\begin{itemize}
			\item from Remark \ref{conjq}, we have $(x,y)\in\mathsf{Q}$ since $|y_{-n}|\neq 0$, for all $n\in\mathbb{Z}_+$;
			
			\item $(x,y)\not\in \mathcal{K}^-$ since $\|(x_{-n},y_{-n})\|_p \xrightarrow{n\to\infty } \infty$.
		\end{itemize}
			
		\item On the other hand, $(x,y)=(-1,-1)\in \mathsf{C}_0\cap \mathcal{K}^{-}$. In fact, 
		 $\max\{|x|,|y|\}=1$ and $(-1,-1)$ is a periodic point of $f$.		
	\end{itemize}
\end{remark}


\section{Proof of Theorem \ref{theokmenosmenor}} \label{provakmenosmenor}

Here, we want to prove that $$Z\subset \mathcal{K}^-\subset R\cup Z$$ where 
$$Z=\{(x,y)\in\mathsf{Q}:|x|=|y|=1\} \textrm{ and } R=\{(x,y)\in\mathsf{Q}:|x|=|y|=|c|\}.$$

Observe that $Z\subset \mathcal{K}^-$ since $Z$ is a bounded set and $f^{-1}(Z)\subset Z$. Furthermore, $R\cap \mathcal{K}^-\neq \emptyset$ since $\alpha=(\alpha_1,\alpha_1)\in R\cap \mathcal{K}^-$.

For each positive integer $n\in \mathbb{Z}_+$, let $f^{-n}(x,y)=(x_{-n},y_{-n})$. Observe that $f^{-n}(x,y)$ exists when $y_{-n+1}\neq 0$, that is $x_{-n}\neq 0$ for all $n \in\mathbb{Z}_+$.

For proving Theorem \ref{theokmenosmenor}, we consider the convenient partition of $\mathsf{Q}$ given by $$\mathsf{Q}=A\cup B \cup P \cup R \cup Z$$
where
$$B=\{(x,y)\in\mathsf{Q}:|x|>1 \textrm{ and } |y|>1\},$$
$$P=\{(x,y)\in\mathsf{Q}:|x|<1 \textrm{ and } |y|<1\}\setminus R$$
$$\textrm{and}$$
$$A=\mathsf{Q} \setminus (B\cup P \cup R \cup Z)$$
and from the following lemmas, we prove that $A\cup B\cup P \subset \mathsf{Q}\setminus \mathcal{K}^-$.

Let $\|\cdot\|$ the $p$-adic norm on $\mathbb{Q}_p^2$ defined by $$\|(x,y)\|:=\max\{ |x|, |y|\}, \textrm{ for all } (x,y)\in \mathbb{Q}_p^2$$ where $|\cdot |$ is the $p$-adic norm on $\mathbb{Q}_p$.

\begin{lemma} \label{lemmaA}
	If $(x,y)\in A$ then $\displaystyle \lim_{n\to -\infty} \|f^n(x,y)\|=+\infty$.
\end{lemma}
\begin{proof} Firstly, observe that $A=A_1\cup A_2 \cup A_3 \cup A_4 \cup A_5 \cup A_6$	where
	
	\begin{itemize}
		\item $A_1=\{(x,y)\in \mathsf{Q} : |x|\leq |c| \textrm{ and } |y|\geq 1 \}$,
		
		\item $A_2=\{(x,y)\in \mathsf{Q} : |x|\geq 1 \textrm{ and } |y|\leq |c| \}$,
		
		\item $A_3=\{(x,y)\in \mathsf{Q} : |x|>1 \textrm{ and } |y|=1 \}$,
		
		\item $A_4=\{(x,y)\in \mathsf{Q} : |x|=1 \textrm{ and } |y|>1 \}$,
		
		\item $A_5=\{(x,y)\in \mathsf{Q} : |x|\geq 1 \textrm{ and } |c|<|y|<1 \}$,
		
		\item $A_6=\{(x,y)\in \mathsf{Q} : |c|<|x|<1 \textrm{ and } |y|\geq 1 \}$.
	\end{itemize}

	It is not hard to prove that $f^{-1}(A_1)\subset A_2$, $f^{-1}(A_2)\subset A_1$, $f^{-1}(A_3)\subset A_4$, $f^{-1}(A_4)\subset A_2\cup A_5$, $f^{-1}(A_5)\subset A_6$, $f^{-1}(A_6)\subset A_2\cup A_5$ and $f^{-2}(A_5)\subset A_2$. Hence, without loss of generality, let us consider $(x,y)\in A_1$.
	
	Since $$|x_{-n-1}|=|y_{-n}| \textrm{ and } |y_{-n-1}|=\dfrac{|x_{-n}-c|}{|y_{-n}|}$$
	it follows that
	$$|x|\leq |c| \textrm{ and } |y|\geq 1,$$ $$|x_{-1}|\geq 1 \textrm{ and } |y_{-1}|\leq |c|,$$ $$|x_{-2}|\leq |c| \textrm{ and } |y_{-2}|\geq |c|^{-1}.$$
	Thus, continuing in this way, we have
	$$|x_{-2i}|\leq |c|^{K_{2i-2}} \textrm{ and } |y_{-2i-2}|\geq |c|^{-K_{2i-1}}$$
	and
	$$|x_{-2i-1}|\geq |c|^{-K_{2i-1}} \textrm{ and } |y_{-2i-1}|\leq |c|^{K_{2i}}$$
	for all $i\geq 1$, where $K_0=K_1=1$ and $$K_{2i}=K_{2i-1}+1 \textrm{ and } K_{2i+1}=K_{2i}+K_{2i-1}$$
	for all $i\geq 1$.
	
	Since $|c|<1$ and $K_n$ goes to infinity when $n$ also go, we finished the proof of Lemma.
\end{proof}

\begin{remark} \label{fib}
	If $\beta=\frac{1+\sqrt{5}}{2}$ is the golden ratio and $(F_n)_{n\geq 0}$ is the Fibonacci sequence, that is $F_0=F_1=1$ and $F_n=F_{n-1}+F_{n-2}$, for all $n\geq 2$, then it is easy to check that
	$$\dfrac{F_{2n+1}}{F_{2n}} < \dfrac{F_{2n+3}}{F_{2n+2}}<\beta <\dfrac{F_{2n+4}}{F_{2n+3}} < \dfrac{F_{2n+2}}{F_{2n+1}}, \;\; \textrm{ for all } n\geq 0.$$
\end{remark}

\begin{lemma} \label{lemmaB}
	If $(x,y)\in B$ then $\displaystyle \lim_{n\to -\infty} \|f^n(x,y)\|=+\infty$.
\end{lemma}
\begin{proof} Firstly, observe that $B=B_1 \cup B_2$ where
	
	\begin{itemize}
		\item $B_1=\{(x,y)\in \mathsf{Q} : |x|>1 \textrm{ and } 1<|y|<|x|^{1/\beta}\}$,
		
		\item $B_2=\{(x,y)\in \mathsf{Q} : |x|>1 \textrm{ and } |y|>|x|^{1/\beta}\}$
	\end{itemize}
	since there no exist $(x,y)\in \mathbb{Q}_p^2$ such that $|y|=|x|^{1/\beta}$.
	
	It is not hard to prove that 
	\begin{center}
		$f^{-1}(B_1)\subset B_2$ and $f^{-1}(B_2)\subset B_1\cup A_2\cup A_3 \cup A_5$
	\end{center} 
	where $A_2$, $A_3$ and $A_5$ are defined in the proof of Lemma \ref{lemmaA}
	
	Let $(x,y)\in B$. If $f^{-i}(x,y)\in A_2\cup A_3 \cup A_5$ for some $i\geq 1$ then we are done by Lemma \ref{lemmaA}. Hence, let us suppose that $f^{-n}(x,y)\in B$ for all $n\in \mathbb{Z}_+$ and without loss of generality, let $(x,y)\in B_2$.
	
	Thus, $|x|>1$, $|y|>|x|^{1/\beta}$ and from Remark \ref{fib}, there exist $n\in\mathbb{Z}_+$ such that 
	$$|y|\geq |x|^{F_{2n}/F_{2n+1}}.$$
	Hence,
	$$|x_{-1}|=|y|>1  \textrm{ and } 1<|y_{-1}|=\dfrac{|x|}{|y|}\leq |y|^{(F_{2n+1}/F_{2n})-1}=|x_{-1}|^{F_{2n-1}/F_{2n}},$$
	$$|x_{-2}|=|y_{-1}|>1  \textrm{ and } |y_{-2}|=\dfrac{|x_{-1}|}{|y_{-1}|}\geq |y|^{(F_{2n}/F_{2n-1})-1}=|x_{-2}|^{F_{2n-2}/F_{2n-1}}$$
	and continuing in this way, there exist $k\in\mathbb{Z}_+$ such that
	$$|x_{-k}|>1  \textrm{ and } 1<|y_{-k}|\leq |x_{-k}|^{F_{1}/F_{2}}=|x_{-k}|^{1/2}$$
	$$\textrm{or}$$
	$$|x_{-k}|>1  \textrm{ and } |y_{-k}| \geq |x_{-k}|^{F_{0}/F_{1}}=|x_{-k}|$$
	that is
	$$|x_{-k-1}|>1  \textrm{ and } |y_{-k-1}|\geq |x_{-k-1}| \implies |x_{-k-2}|>1  \textrm{ and } 1<|y_{-k-2}|\leq 1 $$
	$$\textrm{or}$$
	$$|x_{-k-1}|>1  \textrm{ and } 1<|y_{-k-1}|\leq 1 $$
	which is a contradiction.
	
	Therefore, there exists $k\in \mathbb{Z}_+$ such that $f^{-k}(x,y)\in A_2\cup A_3 \cup A_5$ and from Lemma \ref{lemmaA} we are done.
\end{proof}

\begin{lemma}
	If $(x,y)\in P$ then $\displaystyle \lim_{n\to -\infty} \|f^n(x,y)\|=+\infty$.
\end{lemma}
\begin{proof} Firstly, observe that $P=P_1 \cup P_2 \cup P_3 \cup P_4 \cup P_5\cup P_6$ where
	
	\begin{itemize}
		\item $P_1=\{(x,y)\in \mathsf{Q} : |x|<|c| \textrm{ and } |y|\leq |c|\}$,
		
		\item $P_2=\{(x,y)\in \mathsf{Q} : |c|<|x|<1 \textrm{ and } |y|\leq |c|\}$,
		
		\item $P_3=\{(x,y)\in \mathsf{Q} : |x|<|c| \textrm{ and } |c|<|y| <1 \}$,
		
		\item $P_4=\{(x,y)\in \mathsf{Q} : |c|<|x|<1 \textrm{ and } |c|<|y|<|x|^{1/\beta}\}$,
		
		\item $P_5=\{(x,y)\in \mathsf{Q} : |c|<|x|<1 \textrm{ and } |x|^{1/\beta}<|y|<1\}$,
		
		\item $P_6=\{(x,y)\in \mathsf{Q} : |x|=|c| \textrm{ and } |y|<|c| \textrm{ or } |c|<|y|< 1\}$
	\end{itemize}
	since there no exist $(x,y)\in \mathbb{Q}_p^2$ such that $|y|=|x|^{1/\beta}$.
	
	It is not hard to prove that
	$$\begin{array}{llll}
		f^{-1}(P_1\cup P_2)\subset A_1, &&& f^{-1}(P_3)\subset P_4\cup P_5, \\
		f^{-1}(P_4)\subset P_5\cup A_6, &&& f^{-1}(P_5)\subset P_4. \\
	\end{array}$$	
	where $A_1$ and $A_6$ are defined in the proof of Lemma \ref{lemmaA}. Hence, if $(x,y)\in P_1 \cup P_2 \cup P_3 \cup P_4 \cup P_5$, it suffices to consider the orbit contained in $P_4\cup P_5$, because otherwise if $f^{-i}(x,y)\in A_1\cup A_6$ for some $i\geq 1$ then we are done by Lemma \ref{lemmaA}.
	Without loss of generality, let $(x,y)\in P_4$.
	
	Thus, $|c|<|x|<1$, $|c|<|y|<|x|^{1/\beta}$ and from Remark \ref{fib}, there exist $n\in\mathbb{Z}_+$ such that 
	$$|c|<|y|\leq |x|^{F_{2n}/F_{2n+1}}.$$
	Hence, since $|c|<|x_{-i}|<1$ for all $i\in \mathbb{Z}_+$, in the same way as in the proof of Lemma \ref{lemmaB}, there exists $k\in\mathbb{Z}_+$ such that $|c|<|x_{-k}|<1$ and
	$$|x_{-k}|^{1/2} = |x_{-k}|^{F_{1}/F_{2}} \leq |y_{-k}|<1$$
	$$\textrm{or}$$
	$$ |c|< |y_{-k}| \leq |x_{-k}|^{F_{0}/F_{1}}=|x_{-k}|$$
	that is
	$$|y_{-k-1}|\leq |x_{-k-1}| \implies  1\leq |y_{-k-2}|< 1 $$
	$$\textrm{or}$$
	$$ 1\leq |y_{-k-1}|< 1 $$
	which is a contradiction.
	
	Therefore, there exists $k\in \mathbb{Z}_+$ such that $f^{-k}(x,y)\in A_1\cup A_6$ and from Lemma \ref{lemmaA} we are done.
	
Finally, consider $(x,y)\in P_6$. Thus $|x|=|c|$ and
	\begin{itemize}
		\item if $|y|<|c|$ then $(x_{-1},y_{-1})\in P_1\cup P_3\cup A_1$
		
		\item if $|c|<|y|<1$ then $(x_{-1},y_{-1})\in P_2\cup P_4\cup P_5$
	\end{itemize}
that is, $\displaystyle \lim_{n\to -\infty} \|f^n(x,y)\|=+\infty$.
\end{proof}


\section{Proof of Theorem \ref{teokmenosmaior}} \label{provateokmenosmaior}

\begin{proof}[Proof of item a) of Theorem \ref{teokmenosmaior}]
\begin{lemma}
	We have $f^{-1}(\mathsf{J_0})\subset{\mathsf{J}_0} \subset \mathcal{K}^-$, where
	$$\mathsf{J}_0 :=\{(x,y)\in \mathsf{Q} : \ |x|<|c|,\ 1<|y|<|c|   \}.$$
\end{lemma}
\begin{proof}
	If $(x,y)\in \mathsf{J}_0$, then $|x_{-1}|=|y|\leq |c|$ and $|y_{-1}|=\left|\frac{x-c}{y}\right|=\frac{|c|}{|y|}$, that is $1<|y_{-1}|<|c| $. Therefore, $f^{-1}(\mathsf{J_0})\subset{\mathsf{J}_0}$ and since $\mathsf{J_0}$ is a bounded subset, it follows that $\mathsf{J}_0 \subset \mathcal{K}^-$.
\end{proof}

Now, let us consider the sets
$$\mathsf{J}_1:=\{(x,y)\in \mathsf{Q} : \ |x|>|c|, \ |y|<|c| \textrm{ and } |y|<|x|<|c||y|\},$$
$$\mathsf{J}_2:=\{(x,y)\in \mathsf{Q} : \ |y|>|c|, \ |x|<|c||y| \textrm{ and } |x|<|y|^2<|c||x|\},$$
and for each $n\geq 1$, let
$$\mathsf{J}_{2n+1}:=\left\{ (x,y)\in \mathsf{Q}: \ |x|^{F_{2n-2}}> |c||y|^{F_{2n-1}}, \ |y|^{F_{2n}}<|c| |x|^{F_{2n-1}} \textrm{ and }  |y|^{F_{2n+1}} <|x|^{F_{2n}} < |c||y|^{F_{2n+1}} \right\}$$
and
$$\mathsf{J}_{2n+2}:=\left\{ (x,y)\in \mathsf{Q}: \ |y|^{F_{2n}}> |c||x|^{F_{2n-1}}, \ |x|^{F_{2n}}<|c| |y|^{F_{2n+1}} \textrm{ and }  |x|^{F_{2n+1}} <|y|^{F_{2n+2}} < |c||x|^{F_{2n+1}} \right\}.$$

\begin{remark} Let $(F_n)_{n\ge 0}$ be the Fibonacci sequence defined by $F_0=1$, $F_1=1$, and
	\[
	F_n = F_{n-1} + F_{n-2}, \textrm{ for all} \quad n>1.
	\]
	
	Consider the Fibonacci matrix
	\[
	M = \begin{pmatrix} 1 & 1 \\ 1 & 0 \end{pmatrix}.
	\]
	
	Setting $F_{-1}=0$, it is well-known that
	\[
	M^n = \begin{pmatrix} F_{n} & F_{n-1} \\ F_{n-1} & F_{n-2} \end{pmatrix}  \textrm{ for all } n\geq 1.
	\]
	
	Taking the determinant on both sides, we have
	\[
	(-1)^n=\det(M^n) = \det \begin{pmatrix} F_{n} & F_{n-1} \\ F_{n-1} & F_{n-2} \end{pmatrix} = F_{n} F_{n-2} - F_{n-1}^2.
	\]
	
	Hence, we obtain the Cassini identity:
	
\begin{equation} \label{fibrelation}
F_{n} F_{n-2} - F_{n-1}^2 = (-1)^n \textrm{ for all } n\geq 1.
\end{equation}
Furthermore,
$$F_{n+1} F_{n-2}=F_{n}F_{n-2} +F_{n-1}F_{n-2}= F_{n-1}^2+(-1)^{n} + F_{n-1}F_{n-2}=F_{n-1}(F_{n-1}+F_{n-2})+(-1)^{n} $$
that is
\begin{equation} \label{fibrelation2}
F_{n+1} F_{n-2} = F_{n}F_{n-1}+(-1)^{n}, \textrm{ for all } n\geq 1.
\end{equation}		
\end{remark}

\begin{remark}
The next two lemmas are important for obtaining a geometric representation of the set $\mathcal{K}^-$ in the $p$-adic norm. See Figures \ref{regiaomjinicial} and \ref{regiaomj} for the representation of these sets in $p$-adic norm.
\end{remark}

\begin{lemma} \label{lemajimpar}
If $|c|>1$ then 
$$\textsf{B}_1:=\mathsf{J}_1=\{(x,y)\in \mathsf{Q} : \ |c|<|x|<|c|^2 \textrm{ and } |c|^{-1}|x|<|y|<|c|\}$$
and $\mathsf{J}_{2n+1}=\mathsf{B}_{2n} \cup \mathsf{B}_{2n+1}$ for all $n\geq 1$, where
$$\mathsf{B}_{2n}=\{(x,y)\in \mathsf{Q} : \ |c|^{F_{2n}}<|x|\leq |c|^{F_{2n+1}}  \textrm{ and } (|c|^{-1}|x|^{F_{2n}})^{1/F_{2n+1}} <|y|< (|c|^{-1}|x|^{F_{2n-2}})^{1/F_{2n-1}} \}$$
and
$$\mathsf{B}_{2n+1}=\{(x,y)\in \mathsf{Q} : \ |c|^{F_{2n+1}}<|x|< |c|^{F_{2n+2}} \textrm{ and } (|c|^{-1}|x|^{F_{2n}})^{1/F_{2n+1}} <|y|< (|c||x|^{F_{2n-1}})^{1/F_{2n}}  \}.$$
\end{lemma}
\begin{proof} It is easy to check that
	$$\textsf{B}_1=\{(x,y)\in \mathsf{Q} : \ |c|<|x|<|c|^2 \textrm{ and } |c|^{-1}|x|<|y|<|c|\}.$$

	If $(x,y)\in \mathsf{J}_{2n+1}$, then
$$\begin{array}{rlcrl}
	i)   &  |y|<(|c|^{-1}|x|^{F_{2n-2}})^{1/F_{2n-1}};  &  \;\;\;  &      &   \\
	ii)  &  |y|<(|c||x|^{F_{2n-1}})^{1/F_{2n}};  &   &  iv)  &  |y|>(|c|^{-1}|x|^{F_{2n}})^{1/F_{2n+1}}.  \\
	iii) &  |y|<|x|^{F_{2n}/F_{2n+1}}; & & & \\
\end{array}$$

From $i)$ and $iv)$ we have that $(|c|^{-1}|x|^{F_{2n}})^{1/F_{2n+1}}< (|c|^{-1}|x|^{F_{2n-2}})^{1/F_{2n-1}}$, that is
$$|c|^{-F_{2n-1}}|x|^{F_{2n}F_{2n-1}}< |c|^{-F_{2n+1}}|x|^{F_{2n+1}F_{2n-2}} \implies |c|^{F_{2n+1}-F_{2n-1}}< |x|^{F_{2n+1}F_{2n-2}-F_{2n}F_{2n-1}} .$$
Moreover
\begin{align*}
	F_{2n+1}F_{2n-2}-F_{2n}F_{2n-1}&=(F_{2n}+F_{2n-1})F_{2n-2}-F_{2n}F_{2n-1}\\
	&= F_{2n}F_{2n-2}+F_{2n-1}F_{2n-2}-F_{2n}F_{2n-1} \\
	&= 1+ F_{2n-1}^2 +F_{2n-1}F_{2n-2}-F_{2n}F_{2n-1} \textrm{ (from Relation \ref{fibrelation})} \\
	&= 1+ F_{2n-1} (F_{2n-1}+F_{2n-2})-F_{2n}F_{2n-1} \\
	&= 1+ F_{2n-1}F_{2n}-F_{2n}F_{2n-1} \\
	&= 1. \\
\end{align*}
Thus, from $i)$ and $iv)$ we have $|x|>|c|^{F_{2n}}$.

From $ii)$ and $iv)$ we have  $(|c|^{-1}|x|^{F_{2n}})^{1/F_{2n+1}}< (|c||x|^{F_{2n-1}})^{1/F_{2n}}$, that is,
$$|c|^{-F_{2n}}|x|^{F_{2n}^2}< |c|^{F_{2n+1}}|x|^{F_{2n+1}F_{2n-1}} \implies |x|^{F_{2n}^2-F_{2n+1}F_{2n-1}}< |c|^{F_{2n+1}+F_{2n}} .$$
Since from Relation \ref{fibrelation} we have $F_{2n}^2-F_{2n+1}F_{2n-1}=1$, it follows that $|x|<|c|^{F_{2n+2}}$.

From relations $iii)$ and $iv)$, it follows $|c|^{-1}<1$, which does not give any additional information about $|x|$.

Finally, using Relation \ref{fibrelation}, we can check that the inequalities $i)$, $ii)$ and $iii)$ of $\mathsf{J}_{2n+1}$ satisfy
$$ |y|<(|c|^{-1}|x|^{F_{2n-2}})^{1/F_{2n-1}} \leq |x|^{F_{2n}/F_{2n+1}} \leq (|c||x|^{F_{2n-1}})^{1/F_{2n}}, \textrm{ if }  |c|^{F_{2n}}<|x|\leq |c|^{F_{2n+1}}$$
and
$$|y|<(|c||x|^{F_{2n-1}})^{1/F_{2n}}<|x|^{F_{2n}/F_{2n+1}}<(|c|^{-1}|x|^{F_{2n-2}})^{1/F_{2n-1}}, \textrm{ if }  |c|^{F_{2n+1}}<|x|< |c|^{F_{2n+2}}.$$
Hence, we obtain a partition $\mathsf{B}_{2n} \cup \mathsf{B}_{2n+1}$ of $\mathsf{J}_{2n+1}$ given by
$$\mathsf{B}_{2n}=\{(x,y)\in \mathsf{Q} : \ |c|^{F_{2n}}<|x|\leq |c|^{F_{2n+1}}  \textrm{ and } (|c|^{-1}|x|^{F_{2n}})^{1/F_{2n+1}} <|y|< (|c|^{-1}|x|^{F_{2n-2}})^{1/F_{2n-1}} \}$$
and
$$\mathsf{B}_{2n+1}=\{(x,y)\in \mathsf{Q} : \ |c|^{F_{2n+1}}<|x|< |c|^{F_{2n+2}} \textrm{ and } (|c|^{-1}|x|^{F_{2n}})^{1/F_{2n+1}} <|y|< (|c||x|^{F_{2n-1}})^{1/F_{2n}} \}.$$
\end{proof}

\begin{lemma} \label{lemajpar}
If $|c|>1$ then $\mathsf{J}_{2}=\mathsf{A}_{1} \cup \mathsf{A}_{2}$, where
$$\mathsf{A}_1=\{(x,y)\in \mathsf{Q} : \ |c|<|x|\leq |c|^2  \textrm{ and } |c|<|y|< (|c||x|)^{1/2}\},$$
$$\mathsf{A}_2=\{(x,y)\in \mathsf{Q} : \ |c|^2<|x|< |c|^3  \textrm{ and } |c|^{-1}|x|<|y|< (|c||x|)^{1/2}\}$$
and for all $n\geq 1$, $\mathsf{J}_{2n+2}=\mathsf{A}_{2n+1} \cup \mathsf{A}_{2n+2}$ satisfy
$$\mathsf{A}_{2n+1}=\{(x,y)\in \mathsf{Q} : \ |c|^{F_{2n+1}}<|x|\leq |c|^{F_{2n+2}}  \textrm{ and } (|c||x|^{F_{2n-1}})^{1/F_{2n}} <|y|< (|c||x|^{F_{2n+1}})^{1/F_{2n+2}} \}$$
and
$$\mathsf{A}_{2n+2}=\{(x,y)\in \mathsf{Q} : \ |c|^{F_{2n+2}}<|x|< |c|^{F_{2n+3}} \textrm{ and } (|c|^{-1}|x|^{F_{2n}})^{1/F_{2n+1}} <|y|< (|c||x|^{F_{2n+1}})^{1/F_{2n+2}}  \}.$$
\end{lemma}
\begin{proof}  If $(x,y)\in \mathsf{J}_2$, then
	$$\begin{array}{rlcrl}
		i)   &  |y|>|c|;  &  \;\;\;  &      &   \\
		ii)  &  |y|>|c|^{-1}|x|;  &   &  iv)  &  |y|<(|c||x|)^{1/2}.  \\
		iii) &  |y|>|x|^{1/2}; & & & \\
	\end{array}$$

	From $i)$ and $iv)$ we have  $|c|< (|c||x|)^{1/2}$, that is, $|x|>|c|$.
	
	From $ii)$ and $iv)$ we have  $|c|^{-1}|x|< (|c||x|)^{1/2}$, that is, $|x|<|c|^3$.
	
	It is not necessary to compare $iii)$ and $iv)$.
	
	Finally, we can check that
	$$|y|>|c|\geq |x|^{1/2}\geq |c|^{-1}|x|, \textrm{ if }  |c|<|x|\leq |c|^2$$
	and
	$$|y|>|c|^{-1}|x|>|x|^{1/2}>|c|, \textrm{ if }  |c|^2<|x|< |c|^3.$$
	Hence, we  obtain a partition $\mathsf{A}_1 \cup \mathsf{A}_2$ of $\mathsf{J}_2$ given by
	$$\mathsf{A}_1=\{(x,y)\in \mathsf{Q} : \ |c|<|x|\leq |c|^2  \textrm{ and } |c|<|y|< (|c||x|)^{1/2}\}$$
	and
	$$\mathsf{A}_2=\{(x,y)\in \mathsf{Q} : \ |c|^2<|x|< |c|^3  \textrm{ and } |c|^{-1}|x|<|y|< (|c||x|)^{1/2}\}.$$

	If $(x,y)\in \mathsf{J}_{2n+2}$, then
	$$\begin{array}{rlcrl}
		i)   &  |y|>(|c||x|^{F_{2n-1}})^{1/F_{2n}};  &  \;\;\;  &      &   \\
		ii)  &  |y|>(|c|^{-1}|x|^{F_{2n}})^{1/F_{2n+1}}  &   &  iv)  &  |y|<(|c||x|^{F_{2n+1}})^{1/F_{2n+2}}.  \\
		iii) &  |y|>|x|^{F_{2n+1}/F_{2n+2}}; & & & \\
	\end{array}$$

	From $i)$ and $iv)$ we have $(|c||x|^{F_{2n-1}})^{1/F_{2n}}< (|c||x|^{F_{2n+1}})^{1/F_{2n+2}}$, that is
	$$|c|^{F_{2n+2}}|x|^{F_{2n+2}F_{2n-1}}< |c|^{F_{2n}}|x|^{F_{2n+1}F_{2n}} \implies |c|^{F_{2n+2}-F_{2n}}< |x|^{F_{2n+1}F_{2n}-F_{2n+2}F_{2n-1}} .$$
	Moreover
	\begin{align*}
		F_{2n+1}F_{2n}-F_{2n+2}F_{2n-1}&=F_{2n+1}F_{2n}-(F_{2n+1}+F_{2n})F_{2n-1}\\
		&= F_{2n+1}F_{2n}-F_{2n+1}F_{2n-1}-F_{2n}F_{2n-1}\\
		&= F_{2n+1}F_{2n}+1-F_{2n}^2-F_{2n}F_{2n-1} \textrm{ (from Relation \ref{fibrelation})} \\
		&= 1+ F_{2n} (F_{2n+1}-F_{2n})-F_{2n}F_{2n-1} \\
		&= 1+ F_{2n}F_{2n-1}-F_{2n}F_{2n-1} \\
		&= 1. \\
	\end{align*}
	Thus, from $i)$ and $iv)$ we have $|x|>|c|^{F_{2n+1}}$.

	From $ii)$ and $iv)$ we have that $(|c|^{-1}|x|^{F_{2n}})^{1/F_{2n+1}}< (|c||x|^{F_{2n+1}})^{1/F_{2n+2}}$, that is
	$$|c|^{-F_{2n+2}}|x|^{F_{2n+2}F_{2n}}< |c|^{F_{2n+1}}|x|^{F_{2n+1}^2} \implies |x|^{F_{2n+2}F_{2n}-F_{2n+1}^2}< |c|^{F_{2n+2}+F_{2n+1}} .$$
	Since from Relation \ref{fibrelation} we have $F_{2n+2}F_{2n}-F_{2n+1}^2=1$, it follows that $|x|<|c|^{F_{2n+3}}$.
	
	It is not necessary to compare $iii)$ and $iv)$.
	
	Finally, using Relation \ref{fibrelation}, we can check that the inequalities $i)$, $ii)$ and $iii)$ of $\mathsf{J}_{2n+2}$ satisfying
	$$ |y|>(|c||x|^{F_{2n-1}})^{1/F_{2n}} \geq |x|^{F_{2n+1}/F_{2n+2}} \geq (|c|^{-1}|x|^{F_{2n}})^{1/F_{2n+1}}, \textrm{ if }  |c|^{F_{2n+1}}<|x|\leq |c|^{F_{2n+2}}$$
	and
	$$|y|>(|c|^{-1}|x|^{F_{2n}})^{1/F_{2n+1}} > |x|^{F_{2n+1}/F_{2n+2}} > (|c||x|^{F_{2n-1}})^{1/F_{2n}}, \textrm{ if }  |c|^{F_{2n+2}}<|x|< |c|^{F_{2n+3}}.$$
	Hence, we may have a partition $\mathsf{A}_{2n+1} \cup \mathsf{A}_{2n+2}$ of $\mathsf{J}_{2n+2}$ given by
	$$\mathsf{A}_{2n+1}=\{(x,y)\in \mathsf{Q} : \ |c|^{F_{2n+1}}<|x|\leq |c|^{F_{2n+2}}  \textrm{ and } (|c||x|^{F_{2n-1}})^{1/F_{2n}} <|y|< (|c||x|^{F_{2n+1}})^{1/F_{2n+2}} \}$$
	and
	$$\mathsf{A}_{2n+2}=\{(x,y)\in \mathsf{Q} : \ |c|^{F_{2n+2}}<|x|< |c|^{F_{2n+3}} \textrm{ and } (|c|^{-1}|x|^{F_{2n}})^{1/F_{2n+1}} <|y|< (|c||x|^{F_{2n+1}})^{1/F_{2n+2}}  \}.$$
\end{proof}

\begin{lemma} \label{Jsubset}
	We have $f^{-1}(\mathsf{J}_i)\subset \mathsf{J}_{i-1}$ for all $i\geq 1$. In particular,
	$$\mathsf{J}:=\bigcup_{i=0}^{+\infty}\mathsf{J}_i \subset \bigcup_{i=0}^{+\infty}f^i(\mathsf{J}_0)\subset \mathcal{K}^-.$$
\end{lemma}
\begin{proof} It is not hard to check that  $f^{-1}(\mathsf{J}_2)\subset \mathsf{J}_{1}$ and $f^{-1}(\mathsf{J}_1)\subset \mathsf{J}_{0}$.
	
Let $n\geq 1$. If $(x,y)\in\mathsf{J}_{2n+2}$, then
$$|y|^{F_{2n}}> |c||x|^{F_{2n-1}}, \ |x|^{F_{2n}}<|c| |y|^{F_{2n+1}} \textrm{ and }  |x|^{F_{2n+1}} <|y|^{F_{2n+2}} < |c||x|^{F_{2n+1}}.$$
Since $(x_{-1},y_{-1}):=f^{-1}(x,y)=\left(y,\frac{x-c}{y }\right)$ and from Lemma \ref{lemajpar}, we have $|x|>|c|$, it follows that $|y_{-1}|=\frac{|x|}{|y|}$.

Hence
$$|y|^{F_{2n}}> |c||x|^{F_{2n-1}} \implies |x_{-1}|^{F_{2n-2}}=|y|^{F_{2n-2}} > |c|\left|\frac{x}{y}\right|^{F_{2n-1}}= |c||y_{-1}|^{F_{2n-1}},$$
$$|x|^{F_{2n}}<|c| |y|^{F_{2n+1}} \implies |y_{-1}|^{F_{2n}}=\left|\frac{x}{y}\right|^{F_{2n}}<|c| |y|^{F_{2n-1}}=|c| |x_{-1}|^{F_{2n-1}}$$
and
$$|x|^{F_{2n+1}} <|y|^{F_{2n+2}} < |c||x|^{F_{2n+1}}\implies |y_{-1}|^{F_{2n+1}} =\left|\frac{x}{y}\right|^{F_{2n+1}} <\underbrace{|y|^{F_{2n}}}_{|x_{-1}|^{F_{2n}}} <|c|\left|\frac{x}{y}\right|^{F_{2n+1}}= |c||y_{-1}|^{F_{2n+1}}$$
Therefore, $f^{-1}(x,y)\in\mathsf{J}_{2n+1}$.

Analogously, it is proved that $f^{-1}(\mathsf{J}_{2n+1})\subset\mathsf{J}_{2n}$.
\end{proof}

Let us consider the sets $\mathsf{F}=\{(x,y)\in \mathsf{Q}: |x|< |c| \textrm{ and } |y|<1\}$,
$$\mathsf{G}=\{(x,y)\in \mathsf{Q}: |x|\leq |c| \textrm{ and } |y|>|c|\} \textrm{ and } \mathsf{H}=\{(x,y)\in \mathsf{Q}: |x|> |c| \textrm{ and } |y|\leq 1\}.$$

See Figure \ref{regiaomjinicial} for the representation of these sets in $p$-adic norm.

\begin{lemma} \label{fatoufgh}
	If $|c|>1$, then $(\mathsf{F}\cup \mathsf{G}\cup \mathsf{H} ) \subset (\mathsf{Q}\setminus \mathcal{K}^-)$.
\end{lemma}
\begin{proof} It is not hard to prove that $f^{-1}(\mathsf{F})\subset \mathsf{G}$.
	
	Furthermore, we have that $f^{-1}(\mathsf{G})\subset \mathsf{H}$, $f^{-1}(\mathsf{H})\subset \mathsf{G}$ and $(\mathsf{G}\cup \mathsf{H}) \subset (\mathsf{Q}\setminus \mathcal{K}^-)$. In fact, if $(x,y)\in \mathsf{G}$ then
	$$(x_{-1},y_{-1})=\left(y,\frac{x-c}{y}\right), \textrm{ that is } |x_{-1}|>|c| \textrm{ and } |y_{-1}|\leq 1 $$
	that is $(x_{-1},y_{-1})\in \mathsf{H}$.
	
	Also, if $(x,y)\in \mathsf{H}$, then $|x_{-1}|=|y|\leq 1$ and $|y_{-1}|=\left|\frac{x-c}{y}\right| >|c|$, that is $(x_{-1},y_{-1})\in \mathsf{G}$.
	
	Now, set $|c|=p^d$ and let $(x,y)\in \mathsf{G}$ with $|x|=p^a$ and $|y|=p^b$, where $a,b,d\in\mathbb{Z}_+$ with $d>0$. Thus, since $(x,y)\in \mathsf{G}$, it follows that $a\leq d <b$ and $$
	\begin{array}{lclcclcl}
		|x_{-1}|=p^b & \textrm{and} & |y_{-1}| \leq p^{d-b}, & & & |x_{-2}| \leq p^{d-b} & \textrm{and} & |y_{-2}|\geq p^{2b-d} ,\\
		|x_{-3}|\geq p^{2b-d} & \textrm{and} & |y_{-3}| \leq p^{2d-2b}, & & & |x_{-4}| \leq  p^{2d-2b} & \textrm{and} & |y_{-4}|\geq p^{4b-3d} ,\\
		|x_{-5}|\geq p^{4b-3d} & \textrm{and} & |y_{-5}| \leq p^{4d-4b}, & & & |x_{-6}| \leq  p^{4d-4b} & \textrm{and} & |y_{-6}|\geq p^{8b-7d}.\\
	\end{array}
	$$
	Continuing in this way, we have that
	$$
	\begin{array}{lclcclcl}
		|x_{-2i-1}|\geq p^{2^ib-(2^i-1)d} & \textrm{ and } & |y_{-2i-1}| \leq p^{2^id-2^ib}, \\
		&      & \\
		|x_{-2i-2}| \leq  p^{2^id-2^ib} & \textrm{ and }  & |y_{-2i-2}|\geq p^{2^{i+1}b-(2^{i+1}-1)d},
	\end{array}
	$$
	for all $i\geq 0$. Hence, $\|(x_{-n},y_{-n})\| \geq p^{2^{\lfloor n/2 \rfloor}(b-d)+d}$ which goes to infinity when $n\to +\infty$, since $b-d>0$.
	Therefore, $(x,y)\in \mathsf{Q}\setminus \mathcal{K}^-$.
\end{proof}

Set $F_{-1}=0$. For each $i\geq 0$, let us consider the sets $\mathsf{M}_i$ defined by
$$\mathsf{M}_{2n+1}:=\{(x,y)\in\mathsf{Q}: |x|>|c|^{F_{2n}}, \; |c|^{F_{2n-1}}<|y|\leq |c|^{F_{2n+1}}   \textrm{ and } |y|<(|c|^{-1}|x|^{F_{2n}})^{1/F_{2n+1}} \}$$
and
$$\mathsf{M}_{2n+2}:=\{(x,y)\in\mathsf{Q}: |c|^{F_{2n+1}}<|x|\leq |c|^{F_{2n+3}}   \textrm{ and } |y|>(|c||x|^{F_{2n+1}})^{1/F_{2n+2}} \}$$
for all $n\geq 0$.

See Figures \ref{regiaomjinicial} and \ref{regiaomj} for the representation of these sets in $p$-adic norm.

\begin{figure}[!h]
	\centering
	\includegraphics[scale=0.17]{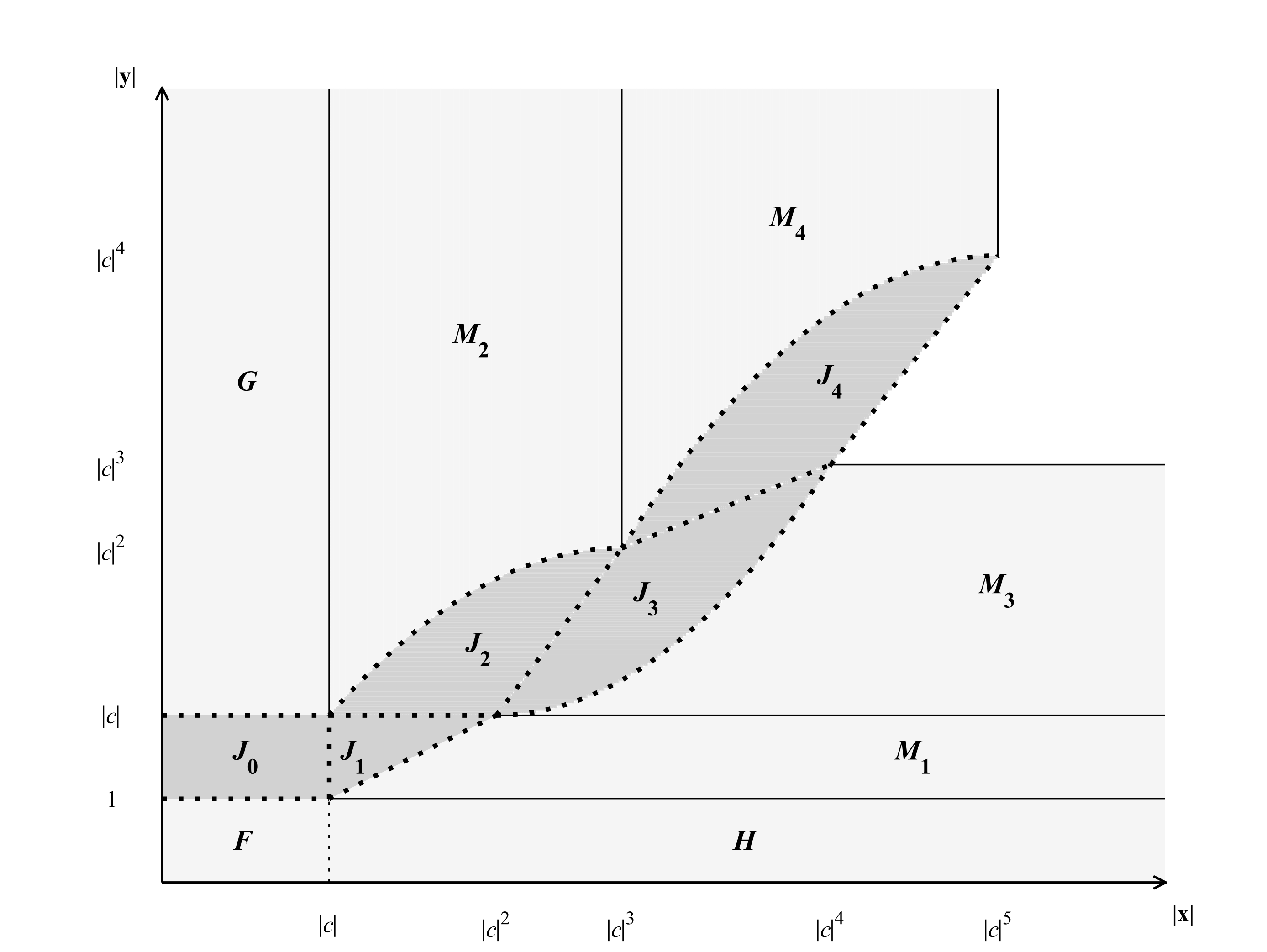}
	\caption{Representation of the sets $\mathsf{F}$, $\mathsf{G}$, $\mathsf{H}$, $\mathsf{J}_i$ and $\mathsf{M}_i$ in $p$-adic norm, for $|c|>1$ and $i=1,2,3,4$.} \label{regiaomjinicial}
\end{figure} 
\begin{figure}[!h]
	\centering
	\includegraphics[scale=0.17]{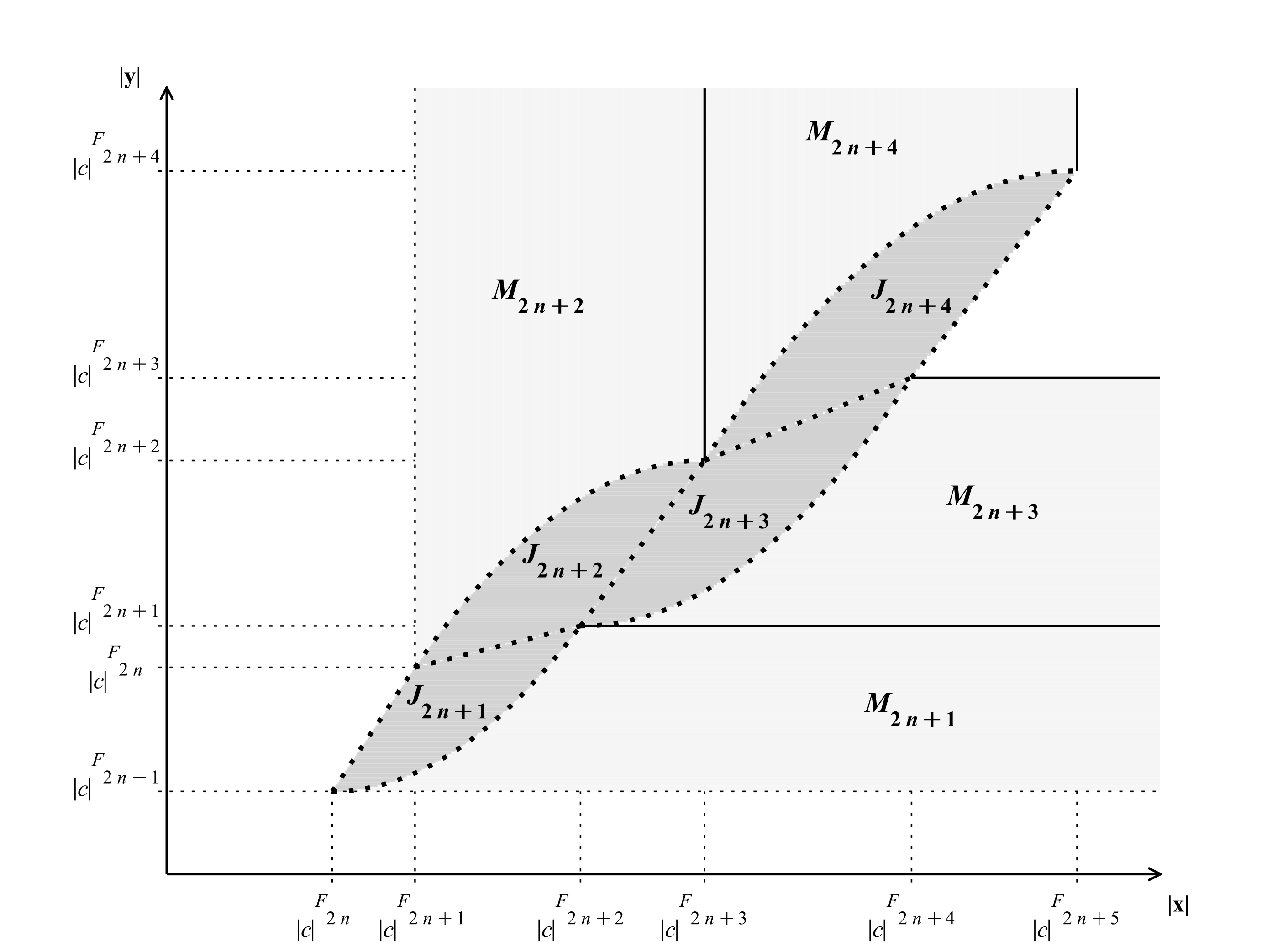}
	\caption{Representation of $\mathsf{J}_i$ and $\mathsf{M}_i$ in $p$-adic norm, for $|c|>1$.} \label{regiaomj}
\end{figure} 

\begin{lemma}
	We have $f^{-1}(\mathsf{M}_2)\subset \mathsf{M}_1$, $f^{-1}(\mathsf{M}_1)\subset \mathsf{G}$ and
$$f^{-1}(\mathsf{M}_{2n+2})\subset \mathsf{H}\cup \bigcup_{i=0}^n\mathsf{M}_{2i+1} \; \textrm{ and } \; f^{-1}(\mathsf{M}_{2n+1})\subset \mathsf{M}_{2n}, \; \textrm{ for all } n\geq 1$$
that is $$\mathsf{M}:=\bigcup_{i=1}^{+\infty} \mathsf{M}_i \subset \mathsf{Q}\setminus\mathcal{K}^-.$$
\end{lemma}
\begin{proof} If $(x,y)\in \mathsf{M}_{1}$, then
$$ |x|>|c|, \; 1<|y|\leq |c|  \textrm{ and } |y|<|c|^{-1}|x|.$$
Thus,
$$ |x_{-1}|=|y|\leq |c| \textrm{ and } |y_{-1}|=\frac{|x|}{|y|}> |c|$$
that is, $f^{-1}(x,y)\in \mathsf{G}$.

If $(x,y)\in \mathsf{M}_{2}$, then
$$ |c|<|x|\leq |c|^3 \textrm{ and } |y|>(|c||x|)^{1/2}.$$
Thus,
$$ |x_{-1}|=|y|>(|c||x|)^{1/2}>|c|,$$
$$|x|<|c|^{-1}|y|^2 \implies |y_{-1}|=\frac{|x|}{|y|}< |c|^{-1}|y|=|c|^{-1}|x_{-1}|$$
and
$$|y_{-1}|=\frac{|x|}{|y|}<\frac{|x|}{(|c||x|)^{1/2}}= \left(\frac{|x|}{|c|}\right)^{1/2}\leq \left(\frac{|c|^3}{|c|}\right)^{1/2}=|c|$$
that is, $f^{-1}(x,y)\in \mathsf{H}\cup \mathsf{M}_1$.

Now, let $n\geq 1$. If $(x,y)\in \mathsf{M}_{2n+2}$, then
$$|c|^{F_{2n+1}}<|x|\leq |c|^{F_{2n+3}}   \textrm{ and } |y|>(|c||x|^{F_{2n+1}})^{1/F_{2n+2}}.$$
Thus,
$$|x_{-1}|=|y|>(|c||c|^{F^2_{2n+1}})^{1/F_{2n+2}} = |c|^{F_{2n}}$$ 
since  $F^2_{2n+1}= F_{2n+2}F_{2n}-1$, by Relation \ref{fibrelation}, 
$$|y_{-1}|=\frac{|x|}{|y|}<\left(\frac{|x|^{F_{2n+2}}}{|c||x|^{F_{2n+1}}}\right)^{1/F_{2n+2}}= (|c|^{-1}|x|^{F_{2n}})^{1/F_{2n+2}}\leq  (|c|^{-1}|c|^{F_{2n+3}F_{2n}})^{1/F_{2n+2}}=|c|^{F_{2n+1}}$$
since  $F_{2n+3}F_{2n}=F_{2n+2}F_{2n+1}+1$, by Relation \ref{fibrelation2}, and
$$|x|^{F_{2n+1}}<|c|^{-1}|y|^{F_{2n+2}} \implies  \left(\frac{|x|}{|y|}\right)^{F_{2n+1}}<|c|^{-1}|y|^{F_{2n}} \implies |y_{-1}|<(|c|^{-1}|x_{-1}|^{F_{2n}})^{1/F_{2n+1}} $$
since $|y_{-1}|= \frac{|x|}{|y|}$ and $|x_{-1}|=|y|$. Therefore, $f^{-1}(x,y)\in \mathsf{H}\cup \bigcup_{i=0}^n\mathsf{M}_{2n+1}$.

Finally, if $(x,y)\in \mathsf{M}_{2n+1}$, then
$$|x|>|c|^{F_{2n}}, \; |c|^{F_{2n-1}}<|y|\leq |c|^{F_{2n+1}}   \textrm{ and } |y|<(|c|^{-1}|x|^{F_{2n}})^{1/F_{2n+1}}.$$
Thus,
$$ |c|^{F_{2n-1}}<|y|=|x_{-1}|\leq |c|^{F_{2n+1}} $$ 
and
$$|x|^{F_{2n}}>|c||y|^{F_{2n+1}} \implies  \left(\frac{|x|}{|y|}\right)^{F_{2n}}>|c||y|^{F_{2n-1}} \implies |y_{-1}|>(|c||x_{-1}|^{F_{2n-1}})^{1/F_{2n}} $$
since $|y_{-1}|= \frac{|x|}{|y|}$ and $|x_{-1}|=|y|$. Therefore, $f^{-1}(x,y)\in \mathsf{M}_{2n}$.
\end{proof}

By construction, let us consider a partition of $\mathsf{Q}\setminus (\mathsf{F}\cup \mathsf{G}\cup \mathsf{H} \cup \mathsf{J} \cup \mathsf{M})$ given by
$$\mathsf{C}=\bigcup_{i=0}^{+\infty} (\mathsf{C}_i \cup \mathsf{D}_{i+2}),$$ where
$$\mathsf{C}_0:=\{(x,y)\in\mathsf{Q}: |x|<|c| \textrm{ and } (|y|=1 \textrm{ or } |y|=|c|)\} \cup \{(x,y)\in\mathsf{Q}: |x|=|c| \textrm{ and } |y|\leq|c|\},$$

$$\mathsf{C}_{2n+1}:=\{(x,y)\in\mathsf{Q}: |c|^{F_{2n}}<|x|\leq|c|^{F_{2n+2}} \textrm{ and } |y|=(|c|^{-1}|x|^{F_{2n}})^{1/F_{2n+1}} \}$$
and 
$$\mathsf{C}_{2n+2}:=\{(x,y)\in\mathsf{Q}: |c|^{F_{2n+1}}<|x|\leq|c|^{F_{2n+3}} \textrm{ and } |y|=(|c||x|^{F_{2n+1}})^{1/F_{2n+2}} \},$$
for $n\geq 0$,
$$\mathsf{D}_2:=\{(x,y)\in\mathsf{Q}: |c|<|x|<|c|^2 \textrm{ and } |y|=|c|\},$$

$$\mathsf{D}_{2n+1}:=\{(x,y)\in\mathsf{Q}: |c|^{F_{2n}}<|x|<|c|^{F_{2n+1}} \textrm{ and } |y|=(|c|^{-1}|x|^{F_{2n-2}})^{1/F_{2n-1}} \}$$
and
$$\mathsf{D}_{2n+2}:=\{(x,y)\in\mathsf{Q}: |c|^{F_{2n+1}}<|x|<|c|^{F_{2n+2}} \textrm{ and } |y|=(|c||x|^{F_{2n-1}})^{1/F_{2n}} \}$$
for $n\geq 1$.

In fact, for example, if $(x,y)\in\mathsf{Q}$ with $ |c|^{F_{2n+1}}<|x|\leq|c|^{F_{2n+2}}$, then:

\begin{itemize}	
	\item if $|y|< (|c|^{-1}|x|^{F_{2n}})^{1/F_{2n+1}}$, then $(x,y)\in \mathsf{H} \cup \bigcup_{i=0}^n \mathsf{M}_{2i+1}$;
	
	\item if $|y|=(|c|^{-1}|x|^{F_{2n}})^{1/F_{2n+1}}$, then $(x,y)\in \mathsf{C}_{2n+1}$;
	
	\item if $(|c|^{-1}|x|^{F_{2n}})^{1/F_{2n+1}} <|y|< (|c||x|^{F_{2n-1}})^{1/F_{2n}} $, then $(x,y)\in \mathsf{B}_{2n+1} \subset \mathsf{J}_{2n+1}$;
	
	\item if $|y|= (|c||x|^{F_{2n-1}})^{1/F_{2n}} $, then $(x,y)\in \mathsf{D}_{2n+2}$;
	
	\item if $(|c||x|^{F_{2n-1}})^{1/F_{2n}} <|y|< (|c||x|^{F_{2n+1}})^{1/F_{2n+2}}$, then $(x,y)\in \mathsf{A}_{2n+1} \subset \mathsf{J}_{2n+2}$;
	
	\item if $|y|= (|c||x|^{F_{2n+1}})^{1/F_{2n+2}}$, then $(x,y)\in \mathsf{C}_{2n+2}$;

	\item if $|y|> (|c||x|^{F_{2n+1}})^{1/F_{2n+2}}$, then $(x,y)\in \mathsf{M}_{2n+2}$.
\end{itemize}

\noindent \textbf{Conclusion:} We have proved that $$\mathsf{Q}= \mathsf{C} \cup \mathsf{J} \cup \mathsf{F} \cup \mathsf{G} \cup \mathsf{H} \cup \mathsf{M} $$
where ($\mathsf{F} \cup \mathsf{G} \cup \mathsf{H} \cup \mathsf{M})\subset (\mathsf{Q}\setminus \mathcal{K}^- )$ and $\mathsf{J}\subset\mathcal{K}^-$, finishing the proof of theorem.
\end{proof}

\begin{proof}[Proof of item b) of Theorem \ref{teokmenosmaior}] From item a) of Theorem \ref{teokmenosmaior}, we have
$$\displaystyle \bigcup_{i=0}^{+\infty} f^n(\mathsf{J_0})\subset \mathcal{K}^- \subset \mathsf{C}\cup \bigcup_{i=0}^{+\infty} f^n(\mathsf{J_0})$$
where
$$\mathsf{J}_0=\{(x,y)\in \mathsf{Q} ;\ |x|<|c|,\ 1<|y|<|c|\}.$$
Since there is no $p$-adic number with $p$-adic norm between $1$ and $p$, it follows that if $|c|=p$, then $J_0=\emptyset$, and the proof is done.
\end{proof}

\begin{proof}[Proof of item c) of Theorem \ref{teokmenosmaior}] Recall that the Haar measure $\mu$ of a closed ball in $\mathbb{Q}_p$ is equal to its radius, that is, $\mu(\overline{B}_r(a)) = r$, where $\overline{B}_r(a)=\{x\in\mathbb{Q}_p: |x-a|_p\leq r\}$. Furthermore, the Haar measure $\mu_2$ on $(\mathbb{Q}_p^2,\mathcal{B}(\mathbb{Q}_p^2))$, where $\mathcal{B}(\mathbb{Q}_p^2)$ is the Borel $\sigma$-algebra of $\mathbb{Q}_p^2$, agrees with the product measure, that is $\mu_2=\mu \times \mu$.
	
Since $|c|>p$, let $k\geq 2$ such that $|c|=p^k$.

For each integer number $n>0$, let us consider the set $$\mathsf{T}_n:=\{(x,y)\in\mathbb{Q}_p^2: |x|=p^{(k-1)F_{n+1}} \textrm{ and } |y|=p^{(k-1)F_{n}} \}.$$
Thus, it is easy to check that $f^{-1}(\mathsf{T_{n}})\subset \mathsf{T}_{n-1}$, since $(x,y)\in \mathsf{T}_n$ implies $|x|>|c|=p^k$, for all $i\geq 1$.

On the other hand, if $(x,y)\in \mathsf{T}_0$, then $|x|=|y|=p^{k-1}$ and
$$|x_{-1}|=|y|=p^{k-1} \textrm{ and } |y_{-1}|=\frac{|x+c|}{|y|}=\frac{|c|}{|y|}=p, $$
and
$$|x_{-2}|=|y_{-1}|=p \textrm{ and } |y_{-2}|=\frac{|x_{-1}+c|}{|y_{-1}|}=\frac{|c|}{|y_{-1}|}=p^{k-1} .$$
Continuing in this way, we have
$$|y_{-2n}|=|x_{-2n+1}|=p^{k-1} \textrm{ and } |x_{-2n}|=|y_{-2n+1}|=p$$
for all $n\geq 1$.

Therefore, 
 $$\mathsf{T}:=\bigcup_{n=0}^{+\infty}\mathsf{T}_n\subset\mathcal{K}^{-}.$$

Hence, since $(\mathsf{T}_n)_{n\geq 0}$ are disjoint subsets, it follows that
$$\mu_2(\mathcal{K}^-)\geq \mu_2(\mathsf{T})= \sum_{n=0}^{+\infty} \mu_2(\mathsf{T}_n) =\sum_{n=0}^{+\infty}\mu(\overline{B}(0,p^{(k-1)F_{n+1}})) \cdot \mu(\overline{B}(0,p^{(k-1)F_{n}}))$$
that is
$$\mu_2(\mathcal{K}^-)\geq \sum_{n=0}^{+\infty}p^{(k-1)F_{n+1}} \cdot p^{(k-1)F_{n}}= \sum_{n=0}^{+\infty}p^{(k-1)F_{n+2}} =+\infty.$$
\end{proof}


\section{Proof of Theorem \ref{teokmenosum}} \label{provakmenosigual}

Let us consider the sets $\mathsf{F}=\{(x,y)\in \mathsf{Q}: |x|< 1 \textrm{ and } |y|<1\}$,
$$\mathsf{G}=\{(x,y)\in \mathsf{Q}: |x|\leq 1 \textrm{ and } |y|>1\} \textrm{ and } \mathsf{H}=\{(x,y)\in \mathsf{Q}: |x|> 1 \textrm{ and } |y|\leq 1\}.$$

Analogously to Lemma \ref{fatoufgh}, we may prove the following:

\begin{lemma} \label{lemac1fgh}
	If $|c|=1$, then $(\mathsf{F}\cup \mathsf{G}\cup \mathsf{H} ) \subset (\mathsf{Q}\setminus \mathcal{K}^-)$.
\end{lemma}

Set $F_{-2}=1$ and $F_{-1}=0$. For each $i\geq 1$, let us consider the sets $\mathsf{M}_i$ defined by
$$\mathsf{M}_{2n+1}:=\{(x,y)\in\mathsf{Q}: |x|^{F_{2n-2}}>|y|^{F_{2n-1}}, \;  |y|^{F_{2n}}>|x|^{F_{2n-1}} \textrm{ and } |x|^{F_{2n}}\leq |y|^{F_{2n+1}} \}$$
and
$$\mathsf{M}_{2n+2}:=\{(x,y)\in\mathsf{Q}: |y|^{F_{2n}}>|x|^{F_{2n-1}}, \;  |x|^{F_{2n}}>|y|^{F_{2n+1}} \textrm{ and } |y|^{F_{2n+2}}\leq |x|^{F_{2n+1}} \}$$
for all $n\geq 0$.

\begin{lemma} \label{lemac1m}
We have $f^{-1}(\mathsf{M}_1)\subset \mathsf{H}$,
$$f^{-1}(\mathsf{M}_{2n+2})\subset \mathsf{M}_{2n+1} \; \textrm{ and } \; f^{-1}(\mathsf{M}_{2n+1})\subset \mathsf{M}_{2n}, \; \textrm{ for all } n\geq 1$$
	that is $$\mathsf{M}:=\bigcup_{i=1}^{+\infty} \mathsf{M}_i \subset \mathsf{Q}\setminus\mathcal{K}^-.$$
\end{lemma}
\begin{proof} If $(x,y)\in \mathsf{M}_{1}$ then
	$$ |x|>1, \; |y|>1  \textrm{ and } |x|\leq|y|.$$
	Thus,
	$$ |x_{-1}|=|y|> 1 \textrm{ and } |y_{-1}|=\frac{|x|}{|y|}\leq 1$$
	that is, $f^{-1}(x,y)\in \mathsf{H}$.
	
	Now, let $n\geq 1$. If $(x,y)\in \mathsf{M}_{2n+2}$, then
$$|y|^{F_{2n}}>|x|^{F_{2n-1}}, \;  |x|^{F_{2n}}>|y|^{F_{2n+1}} \textrm{ and } |y|^{F_{2n+2}}\leq |x|^{F_{2n+1}}.$$
	Thus,
$$|y|^{F_{2n}}>|x|^{F_{2n-1}} \implies |x_{-1}|^{F_{2n-2}}=|y|^{F_{2n-2}}> \frac{|x|^{F_{2n-1}}}{|y|^{F_{2n-1}}}=|y_{-1}|^{F_{2n-1}},$$ 
$$|x|^{F_{2n}}>|y|^{F_{2n+1}} \implies |y_{-1}|^{F_{2n}}=\frac{|x|^{F_{2n}}}{|y|^{F_{2n}}}>|y|^{F_{2n-1}} =|x_{-1}|^{F_{2n-1}} \; \textrm{ and}$$ 
$$|y|^{F_{2n+2}}\leq |x|^{F_{2n+1}} \implies |x_{-1}|^{F_{2n}}=|y|^{F_{2n}}\leq \frac{|x|^{F_{2n+1}}}{|y|^{F_{2n+1}}}=|y_{-1}|^{F_{2n+1}}$$ 
$f^{-1}(x,y)=(x_{-1},y_{-1})\in \mathsf{M}_{2n+1}$.
	
	Finally, if $(x,y)\in \mathsf{M}_{2n+1}$, in the same way, we may prove that $f^{-1}(x,y)\in \mathsf{M}_{2n}$.
\end{proof}

\begin{remark}
From Remark \ref{fib}, we can rewrite the sets $\mathsf{M}_i$ by
$$\mathsf{M}_{1}:=\{(x,y)\in\mathsf{Q}: 1< |x|\leq |y|\}, \;\; \mathsf{M}_{2}:=\{(x,y)\in\mathsf{Q}: 1 <|y|\leq |x|^{1/2} \},$$
$$\mathsf{M}_{2n+1}:=\{(x,y)\in\mathsf{Q}: |x|^{F_{2n}/F_{2n+1}}\leq |y|< |x|^{F_{2n-2}/F_{2n-1}} \}$$
and
$$\mathsf{M}_{2n+2}:=\{(x,y)\in\mathsf{Q}: |x|^{F_{2n-1}/F_{2n}} <|y|\leq |x|^{F_{2n+1}/F_{2n+2}} \}$$
for all $n\geq 1$.

On the other hand, it is not hard to verify that
$$\mathsf{Q}\setminus (\mathsf{C}_0 \cup \mathsf{F} \cup \mathsf{G} \cup \mathsf{H} \cup \mathsf{M})=\{(x,y)\in \mathsf{Q}: |x|>1 \textrm{ and } |y|=|x|^{1/\beta}\}, $$
where $\beta=\frac{1+\sqrt{5}}{2}$ is the golden ratio,
$$\mathsf{C}_0= \{(x,y)\in \mathsf{Q}: \; \max \{|x|,|y|\}=1\} \; \textrm{ and }\;  \mathsf{M}:=\bigcup_{i=1}^{+\infty}\mathsf{M}_i.$$

Since 
there do not exist $p$-adic numbers $x,y\in \mathbb{Q}_p$ such that  $|y|=|x|^{1/\beta}$, we have constructed a partition of $\mathsf{Q}$ (see Figure \ref{partitioncigual1}) given by
$$\mathsf{Q}= \mathsf{C}_0 \cup \mathsf{F} \cup \mathsf{G} \cup \mathsf{H} \cup \mathsf{M}.$$

Therefore, from Lemmas \ref{lemac1fgh} and \ref{lemac1m} we have ($\mathsf{F} \cup \mathsf{G} \cup \mathsf{H} \cup \mathsf{M})\subset (\mathsf{Q}\setminus \mathcal{K}^- )$ and $\mathcal{K}^- \subset \mathsf{C}_0$, finishing the proof of theorem.
\end{remark}

\begin{figure}[!h]
	\centering
	\includegraphics[scale=0.17]{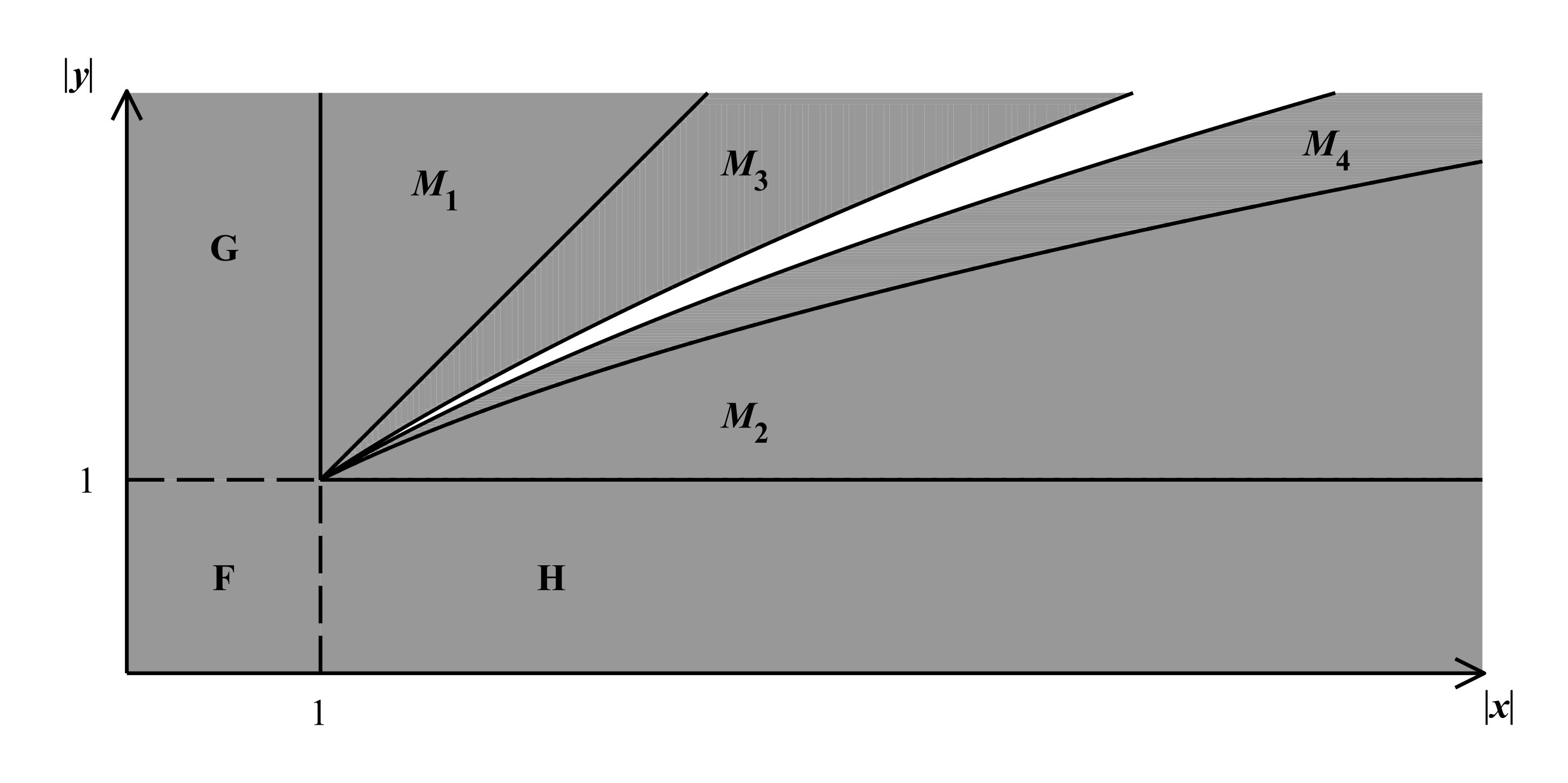}
	\caption{The partition $\mathsf{Q}= \mathsf{C}_0 \cup \mathsf{F} \cup \mathsf{G} \cup \mathsf{H} \cup \mathsf{M}$ in $p$-adic norm, for $|c|=1$.} \label{partitioncigual1}
\end{figure}


\section{Open Problems} \label{open}

In \cite{BCR}, we proved that if $|c|_p < 1$, then a subset of the forward Julia set $\mathcal{K}^+$ of $f$ with infinite Haar measure is the stable manifold of a fixed point of $f$. Moreover, in \cite{BCM} and \cite{C}, the authors studied the backward Julia set $\mathcal{K}^-$ and showed that $\mathcal{K}^-$ is the union of the unstable manifolds of the saddle fixed point and of the $3$-periodic points. These results naturally lead to the following question.

\begin{problem} Is the backward $p$-adic Julia set $\mathcal{K}^-$ an union of unstable manifolds of fixed and periodic points?
\end{problem}

Another interesting question comes from paper \cite{ADP} where the authors proved that their dynamics is topologically conjugate to the two-sided shift map on the space of bisequences in two symbols, which is related to the Smale horseshoe map. Hence, a natural problem is the following.

\begin{problem} For some $c\in\mathbb{Q}_p^2$, is $f$ topologically conjugate to the two-sided shift map on the space of bisequences in two symbols?
\end{problem}

Also, an interesting problem is about the Haar measure of $\mathcal{K}^-$ when $|c|\leq1$. From Theorem \ref{teokmenosmaior}, if $|c|> 1$, then the Haar measure of $\mathcal{K}^-$ is infinity. However, if $|c|\leq1$, then we don't no if $\mathcal{K}^-$ has infinitely many elements (in this case the Haar measure can be positive) or if  $\mathcal{K}^-$ is a finite set (in this case the Haar measure is null). We therefore pose the following problem.

\begin{problem} If $|c|\leq1$, what is the Haar measure of $\mathcal{K}^-$?
\end{problem}

The last interesting problem concerns the filled Julia set $\mathcal{K} = \mathcal{K}^- \cap \mathcal{K}^+$. In \cite{ADP}, the authors considered a class of Hénon maps $H_{a,b}$ depending on two $p$-adic parameters $a,b \in \mathbb{Q}_p$, and established several basic properties of $\mathcal{K}(H_{a,b})$. In particular, they determined, for each pair of parameters $a,b \in \mathbb{Q}_p$, whether $\mathcal{K}(H_{a,b})$ is empty or nonempty, bounded or unbounded, whether it coincides with the unit ball, and, for a certain region of the parameter space, they proved that $\mathcal{K}(H_{a,b})$ is an attractor. This leads us to the following final question.

\begin{problem}
	What properties can be established for the filled Julia set $\mathcal{K} = \mathcal{K}^- \cap \mathcal{K}^+$ of $f$?
\end{problem}

\vspace{1 em}

\noindent {\bf Acknowledgment}: The authors are grateful to Ali Messaoudi for many valuable comments, insightful discussions, and helpful suggestions.


\end{document}